\documentclass[11pt,a4paper]{article}
\usepackage{amsfonts}
\usepackage{latexsym}
\usepackage{cite}
\usepackage{amsmath,amsfonts,latexsym,amssymb,color}
\usepackage[mathscr]{eucal}
\usepackage{cases}
\usepackage{amsthm}

\usepackage[bf,small]{caption2}
\usepackage{float}
\usepackage{graphicx}
\usepackage{amsmath}
\usepackage{amssymb}
\usepackage[all]{xy}

\newtheorem{theorem}{theorem}[section]

\newtheorem{lem}[theorem]{Lemma}

\newtheorem{alg}[theorem]{Algorithm}

\newtheorem{defn}[theorem]{Definition}
\newtheorem{exmp}[theorem]{Example}

\newtheorem{rmk}[theorem]{Remark}
\newtheorem{nota}[theorem]{Notation}
\newtheorem{conv}[theorem]{Convention}
\newtheorem{nota-conv}[theorem]{Notation $\&$ Convention}

\begin{document}

\title{\textbf{Cohomological invariants of representations of 3-manifold groups}}
\author{\Large Haimiao Chen \footnote{Email: \emph{chenhm@math.pku.edu.cn}} \\
\normalsize \em{Mathematics, Beijing Technology and Business University, Beijing, China}}
\date{}
\maketitle

\begin{abstract}
  Suppose $\Gamma$ is a discrete group, and $\alpha\in Z^3(B\Gamma;A)$, with $A$ an abelian group. Given a representation $\rho:\pi_1(M)\to\Gamma$, with $M$ a closed 3-manifold, put $F(M,\rho)=\langle(B\rho)^\ast[\alpha],[M]\rangle$, where $B\rho:M\to B\Gamma$ is a continuous map inducing $\rho$ which is unique up to homotopy, and $\langle-,-\rangle:H^3(M;A)\times H_3(M;\mathbb{Z})\to A$ is the pairing. We extend the definition of $F(M,\rho)$ to manifolds with corners, and establish a gluing law. Based on these, we present a practical method for computing $F(M,\rho)$ when $M$ is given by a surgery along a link $L\subset S^3$. In particular, the Chern-Simons invariant can be computed this way.

\end{abstract}

\section{Introduction}

Suppose $\Gamma$ is a discrete group, and $\alpha\in Z^3(B\Gamma;A)$, with $A$ an abelian group. Given a representation $\rho:\pi_1(M)\to\Gamma$, with $M$ a closed 3-manifold, put
\begin{align}
F(M,\rho)=\langle(B\rho)^\ast[\alpha],[M]\rangle\in A,  \label{eq:def-F}
\end{align}
where $B\rho:M\to B\Gamma$ is a continuous map inducing $\rho$ which is unique up to homotopy,
and $\langle-,-\rangle:H^3(M;A)\times H_3(M;\mathbb{Z})\to A$ is the pairing.
It is a subtle problem to define $F(M,\rho)$ when $\partial M\ne\emptyset$, and will be handled in this paper.

There are at least two reasons for us to care about this cohomological invariant.
First, if $\Gamma$ is a finite group and $A=\mathbb{R}/\mathbb{Z}$, then
$$\frac{1}{\#\Gamma}\sum\limits_{\rho:\pi_1(M)\to\Gamma}\exp\big(\sqrt{-1}F(M,\rho)\big)\in\mathbb{C}$$
is by definition the {\it Dijkgraaf-Witten invariant} of $M$ associated to $[\alpha]$.
Second, if $A=\mathbb{C}/\mathbb{Z}$ and $\Gamma={\rm SL}(n,\mathbb{C})^{\delta}$ viewed as a discrete group, then for a certain $\alpha$ representing the {\it Cheeger-Chern-Simons class} $\hat{C}_2\in H^3(B\Gamma;\mathbb{C}/\mathbb{Z})$, one has that $F(M,\rho)$ equals the {\it Chern-Simons invariant} (CSI for short) ${\rm CS}(\rho)$ which is meant to be that of the flat connection corresponding to $\rho$.

The importance of CSI is manifested in several aspects of geometry and topology; see \cite{DG11,DLW15,Fr95,Fr09,Wi89} and the references therein.
There have been many works in the literature on computing CSI.
Zickert \cite{Zi09} gave a formula for boundary-parabolic ${\rm SL}(2,\mathbb{C})$-representations, for $M$ with tori boundary. Hatakenaka and Nosaka \cite{HN12}, Inoue and Kabaya \cite{IK14} used quandle to derive a new formula for
$G={\rm SL}(2,\mathbb{C})$. March\'e \cite{Ma14} filled all tetrahedra with a connection as explicit as possible, and computed the contribution of each teterhedron. Garoufalidits, Thurston and Zickert \cite{GTZ15} gave a formula for any $M$ and $G={\rm SL}(n,\mathbb{C})$.
Besides, computations for concrete manifolds are seen in: Kirk and Klassen \cite{KK90}; Cho, Murakami and Yokota \cite{CMY09}; Ham and Lee \cite{HL17}.

We aim to present a convenient method for computing general cohomological invariants. It is more flexible in that there needs to be no restriction on boundary.

In Section 2, we set up a general framework for $F(M,\rho)$, and reveal some fine structures around; in particular, we define $F(M,\rho)$ when $M$ is a 3-manifold with boundary. This is largely based on \cite{Ch16} Section 2, which in turn was an exposition using algebraic notions of the construction given by Freed \cite{Fr94}.

In Section 3, we propose a method for computing $F(M,\rho)$ when $M$ is the complement of a link in $S^3$. An efficient procedure is designed, starting from a link diagram. After that, if a closed 3-manifold $M$ is given as a surgery along a link, then $F(M,\rho)$ can be written down. All these are in accord with the sprit of Turaev's
{\it homotopy quantum field theory} \cite{Tu10}.

\medskip

{\bf Acknowledgement} \\
The author is supported by NSFC-11771042.

\section{General framework}

\subsection{Preparation}

\begin{conv}
\rm In this paper, all manifolds are oriented. For a manifold $M$, let $-M$ denote the one obtained by reversing the orientation of $M$.
\end{conv}

For a topological space $X$, let $\Pi_1(X)$ denote the fundamental groupoid of $X$, i.e. the category whose objects are points of $X$ and whose morphisms are homotopy classes of paths. Let $\mathfrak{B}(X)=\mathcal{F}un(\Pi_1(X),\Gamma)$, the set of functors $\Pi_1(X)\to\Gamma$, where $\Gamma$ is viewed as a groupoid with a single object.

Suppose $S$ is a finite subset of $X$ such that each connected component of $X$ contains at least one point from $S$.
Let $\Pi_1^S(X)$ be the full subgroupoid of $\Pi_1(X)$ with $S$ as the set of objects, then $\Pi_1^S(X)$ is equivalent to $\Pi_1(X)$ through the inclusion $S\subset X$.
There is a restriction
$$\mathfrak{B}(X)\to\mathfrak{B}^S(X):=\mathcal{F}un(\Pi^S_1(X),\Gamma).$$
For $\rho\in\mathfrak{B}(X)$, abusing the notation, we denote its image in $\mathfrak{B}^S(X)$ also by $\rho$.
In most situations in this paper, it is sufficient to consider $\mathfrak{B}^S(X)$ for some $S$.

Let $\Delta^k=[v_0,\ldots,v_k]$ denote the standard $k$-simplex. For each singular $k$-simplex $\sigma:\Delta^k\to X$, let $\sigma_i=\sigma(v_i)$ and let $\sigma_{ij}=\sigma|_{[v_i,v_j]}:[v_i,v_j]\to X$. Abbreviate $C_k(X;\mathbb{Z})$ to $C_k(X)$.

Let $\phi\in\mathfrak{B}(X)$. Given a singular $k$-chain $\xi=\sum_in_i\sigma^i\in C_k(X)$, put
\begin{align}
\xi\langle\phi\rangle=\sum_in_i\sigma^i\langle\phi\rangle\in C_k(B\Gamma),
\end{align}
where for each singular $k$-simplex $\sigma\in C_k(B\Gamma)$, we set
$\sigma\langle\phi\rangle=[\phi(\sigma_{01})|\cdots|\phi(\sigma_{k-1,k})]$.

Let $A$ be an abelian group, for which additive notations are used, and let $\Gamma$ be a discrete group. A 3-cocycle $\alpha\in Z^3(B\Gamma;A)=Z^3(\Gamma;A)$ is a function $\Gamma\times\Gamma\times\Gamma\to A$ satisfying
\begin{align*}
\alpha(x,y,z)-\alpha(xy,z,w)+\alpha(x,yz,w)-\alpha(x,y,zw)+\alpha(y,z,w)=0
\end{align*}
for all $x,y,z,w\in\Gamma$.
Recall that (see \cite{Ha02} Page 89) $B\Gamma$ has a simplicial model, in which $k$-simplices are ordered $k$-tuples $[x_1|\cdots|x_k]$ with $x_i\in\Gamma$. The boundary map $\partial_k:C_k(B\Gamma)\to C_{k-1}(B\Gamma)$ is given by
$$\partial[x_1|\cdots|x_k]=[x_2|\cdots|x_k]+\sum\limits_{i=1}^{k-1}(-1)^i
[x_1|\cdots|x_ix_{i+1}|\cdots|x_k]+(-1)^k[x_1|\cdots|x_{k-1}].$$
The value of $\alpha$ taking at $[x|y|z]$ is $\alpha(x,y,z)$.

\subsection{Extending to manifolds with boundary}

\begin{defn}
\rm For an $n$-manifold $M$, call a singular chain $\zeta\in C_n(M)$ an {\it s-triangulation} if $\zeta$ represents the fundamental class $[M,\partial M]$. Let $[M]$ denote the set of s-triangulations. This is consistent with the notion of the fundamental class of $M$ when $M$ is closed.

Viewing $\emptyset$ as an $n$-manifold, put $[\emptyset]=\{0\}$.

Given $\zeta_i\in[M_i],i=1,\ldots,r$, the meaning of $\zeta_1\sqcup\ldots\sqcup\zeta_r\in[M_1\sqcup\ldots\sqcup M_r]$ is self-evident.
\end{defn}

Let $Y$ be a closed surface. For $\tau\in\mathfrak{B}(Y)$, define $F(Y,\tau)$ to be the set of the functions $f:[Y]\to A$ such that $f(\xi_1)-f(\xi_2)=\alpha(\zeta\langle\tau\rangle)$ for any $\xi_1,\xi_2\in[Y]$ and $\zeta\in C_3(Y)$ with $\partial\zeta=\xi_1-\xi_2$.
Notice two facts: (i) for any $\xi_1,\xi_2\in[Y]$, since $\partial(\xi_1-\xi_2)=0$ and $H_3(Y;\mathbb{Z})=0$, one can find $\zeta$ such that $\partial\zeta=\xi_1-\xi_2$; (ii) if $\partial\zeta_1=\partial\zeta_2=\xi_1-\xi_2$, then $\alpha(\zeta_1\langle\tau\rangle)=\alpha(\zeta_2\langle\tau\rangle)$, as $\partial(\zeta_1\langle\tau\rangle-\zeta_2\langle\tau\rangle)=0$ and $\alpha$ is a cocycle.
Hence the action $A\times F(Y,\tau)\to F(Y,\tau)$ given by
$$(a,f)\mapsto a.f, \qquad \text{with} \quad (a.f)(\xi)=a+f(\xi)$$
is free and transitive, so $F(Y,\tau)$ is an {\it $A$-torsor}, as is called in the literature.

For a closed 3-manifold $X$, take an arbitrary $\zeta\in [X]$, and put $F(X,\rho)=\alpha(\zeta\langle\rho\rangle)$ for any $\rho\in\mathfrak{B}(X)$. Clearly,
this is in consistence with (\ref{eq:def-F}).

Now consider a general 3-manifold $X$. Write $X=X^{\rm cl}\sqcup X'$, where $X^{\rm cl}$ is closed, and each connected component of $X'$ has nonempty boundary. Then $\partial X=\partial X'$. Given $\rho\in\mathfrak{B}(X)$, let $\rho^{\rm cl}=\rho|_{X^{\rm cl}}$ and $\rho'=\rho|_{X'}$.
For each $\xi\in[\partial X']$, define $F(X',\rho')(\xi)$ to be $\alpha(\zeta\langle\rho'\rangle)\in A$ for any $\zeta\in[X']$ (so that $\partial\zeta=\xi$). This is well-defined, thanks to the following:
(i) such a $\zeta$ always exists;
(ii) if $\partial\zeta_1=\partial\zeta_2=\xi$, then $\partial(\zeta_1-\zeta_2)=0$, and since $H_3(X';\mathbb{Z})=0$, we may find $\eta\in C_4(X')$ with $\partial\eta=\zeta_1-\zeta_2$, so $\alpha(\zeta_1\langle\rho'\rangle)=\alpha(\zeta_2\langle\rho'\rangle)$.
Moreover, if $\xi_1-\xi_2=\partial\zeta$, then taking $\zeta_2$ with $\partial\zeta_2=\xi_2$ and putting $\zeta_1=\zeta+\zeta_2$, we have
$$F(X',\rho')(\xi_1)-F(X',\rho')(\xi_2)=\alpha(\zeta_1\langle\rho'\rangle)-\alpha(\zeta_2\langle\rho'\rangle)
=\alpha(\zeta\langle\rho'\rangle).$$
Hence indeed $F(X',\rho')\in F(\partial X',\partial\rho')$, with $\partial\rho'=\rho'|_{\partial X'}$.
Define
$$F(X,\rho)=F(X^{\rm cl},\rho^{\rm cl}).F(X',\rho')\in F(\partial X',\partial\rho')=F(\partial X,\partial\rho).$$

\begin{rmk}
\rm We highlight that to determine $F(X,\rho)$, it suffices to take an arbitrary $\xi_0\in[\partial X]$ and specify the value $F(X,\rho)(\xi_0)\in A$. Then the value of $F(X,\rho):[\partial X]\to A$ at any $\xi$ is given by $F(X,\rho)(\xi)=\alpha(\zeta)+F(X,\rho)(\xi_0)$, for an arbitrary $\zeta$ with $\partial\zeta=\xi-\xi_0$.
\end{rmk}

\begin{lem}[Gluing law] \label{lem:glue}
Suppose $\partial X=Y'\sqcup Y_1\sqcup -Y_2$, and $\varphi:Y_1\stackrel{\approx}\longrightarrow Y_2$ is an orientation-preserving homeomorphism. Let $X^\varphi=X/\{y\sim\varphi(y), y\in Y_1\}$, i.e. the manifold obtained from gluing $X$ along $\varphi$, so that $\partial X^\varphi=Y'$;
let ${\rm gl}_\varphi:X\to X^\varphi$ denote the quotient map.
Given $\rho\in\mathfrak{B}(X^\varphi)$, let $\tilde{\rho}=({\rm gl}_\varphi)^\ast(\rho)$, then for any $\xi\in[Y_1]$ and $\xi'\in[Y']$, one has
$$F(X^\varphi,\rho)(\xi')=F(X,\tilde{\rho})(\xi'\sqcup\xi\sqcup-\varphi_\#(\xi)).$$
\end{lem}
\begin{proof}
Take $\zeta\in C_3(X)$ with $\partial\zeta=\xi'\sqcup\xi\sqcup-\varphi_\#(\xi)$, then
$({\rm gl}_\varphi)_\#(\zeta)\in[X^\varphi]$, so
$$F(X^\varphi,\rho)(\xi')=\alpha(\zeta\langle\rho\rangle)=F(X,\tilde{\rho})(\xi'\sqcup\xi\sqcup-\varphi_\#(\xi)).$$
\end{proof}

We must also allow manifolds to have corners. In this paper, 3-manifolds with corners $X$ are viewed as ordinary 3-manifolds together with a piece of information encoding how to decompose $\partial X$ into subsurfaces along circles. For $\rho\in\mathfrak{B}(X)$, we simply define $F(X,\rho)\in F(\partial X,\partial\rho)$ as above, temporarily forgetting that $X$ has corners. A gluing rule in this more general context can be established.
But it is such an easy task that we choose not to explicitly write down.

\subsection{Fundamental cycle}

Motivated by \cite{No20,Ku17,Zi09}, we introduce a universal notion.

Given a topological space $X$, define an equivalence relation on $C_k(X)$ by declaring that two singular $k$-simplices $\sigma, \sigma'$ are equivalent if $\sigma_{ij}=\sigma'_{ij}$ in $\Pi_1(X)$ for all $i,j$, and extending linearly. Let $\overline{C}_k(X)$ denote the set of equivalence classes. Clearly $\partial_k:C_k(X)\to C_{k-1}(X)$ descends to a map $\overline{\partial}_k:\overline{C}_k(X)\to\overline{C}_{k-1}(X)$.
Let $\mathcal{Y}_X$ denote the composite
$$\mathcal{Y}_X:C_3(X)\twoheadrightarrow \overline{C}_3(X)\twoheadrightarrow\overline{C}_3(X)/{\rm Im}(\overline{\partial}_4)=:\grave{C}(X).$$

For a 3-manifold $X$, define $F_X:[\partial X]\to \grave{C}(X)$ by sending $\xi\in[\partial X]$ to $\mathcal{Y}_X(\eta)$, for any $\eta\in[X]$ with $\partial\eta=\xi$.
This is well-defined: such an $\eta$ always exists; if also $\partial\eta'=\xi$, then $\eta'-\eta$ represents $0\in H_3(X)$, so that $\eta'-\eta=\partial\mu$ for some $\mu\in C_4(X)$, hence $\mathcal{Y}_X(\eta)=\mathcal{Y}_X(\eta')$.
Call the map $F_X$ the {\it fundamental cycle} of $X$.
\begin{rmk}
\rm Let $\iota:\partial X\hookrightarrow X$ be the inclusion.
Straightforward is the property that $F_X(\xi')-F_X(\xi)=\mathcal{Y}(\iota_\#(\zeta))$
for any $\xi,\xi'\in[\partial X]$ and $\zeta\in C_3(\partial X)$ with $\partial\zeta=\xi'-\xi$.
\end{rmk}

For any $\rho\in\mathfrak{B}(X)$ and any $\alpha\in Z^3(B\Gamma;A)$, abusing the notation, there is an induced map $(B\rho)_\#:\grave{C}(X)\to C_3(B\Gamma)/{\rm Im}(\partial_4)$ such that the composite
\begin{align*}
[\partial X]\stackrel{F_X}\longrightarrow \grave{C}(X)\stackrel{(B\rho)_\#}\longrightarrow C_3(B\Gamma)/{\rm Im}\partial_4\stackrel{\alpha}\longrightarrow A
\end{align*}
equals $F(X,\rho)$ defined in the previous subsection.

\subsection{Cycle-cocycle calculus}

The name was given in \cite{Ch16}, to mean the following procedure: given $\tau\in\mathfrak{B}(Y)$ and $\xi,\xi'\in[Y]$, to find $\zeta\in C_3(Y)$ with $\partial\zeta=\xi-\xi'$ and compute $\omega^\tau(\xi;\xi'):=\alpha(\zeta\langle\tau\rangle)$, which is independent of the choice of $\zeta$.
We abbreviate $\omega^\tau(\xi;\xi')$ to $\omega(\xi;\xi')$ when $\tau$ is clear.

\begin{conv}
\rm From now on, we assume that $\alpha$ is {\it strongly normalized}, meaning $\alpha(x,y,z)=0$ whenever $1\in\{x,y,z,xy,yz\}$, as defined in \cite{HN12} Section 6.1.
\end{conv}

\begin{rmk}
\rm Such an $\alpha$ can be found at least for the Chern-Simons invariant over ${\rm SL}(2,\mathbb{C})$ (see \cite{HN12} Lemma 6.4), and we expect similar results for more Lie groups.

If $\alpha$ is not strongly normalized, then the formulas below still exist, but will be more complicated. We may develop these in future work.
\end{rmk}

\begin{nota}
\rm For $\mu,\mu\in C_2(Y)$, suppose there exists $\zeta\in C_3(Y)$ with $\partial\zeta=\mu-\mu'$. Denote $\mu\doteq\mu'$ (resp. $\mu\equiv\mu'$) if $\alpha(\zeta\langle\tau\rangle)=0$ for all $\tau\in\mathfrak{B}(Y)$ and all normalized (resp. strongly normalized) $\alpha$.
\end{nota}

\begin{conv} \label{conv:S}
\rm From now on, for each manifold $X$, write $\mathfrak{B}^0(X)$ instead of $\mathfrak{B}^S(X)$, after a finite set $S$ is chosen.

For a planar surface, let $S$ consist of one vertex from each component of $\partial Y$. For torus $\mathbb{T}^2=S^1\times S^1$, let $S=\{1\times 1\}$.
\end{conv}

In Fig. \ref{fig:notation} we introduce some pictorial notations which are used throughout this paper.

\begin{figure}[h]
  \centering
  \includegraphics[width=0.7\textwidth]{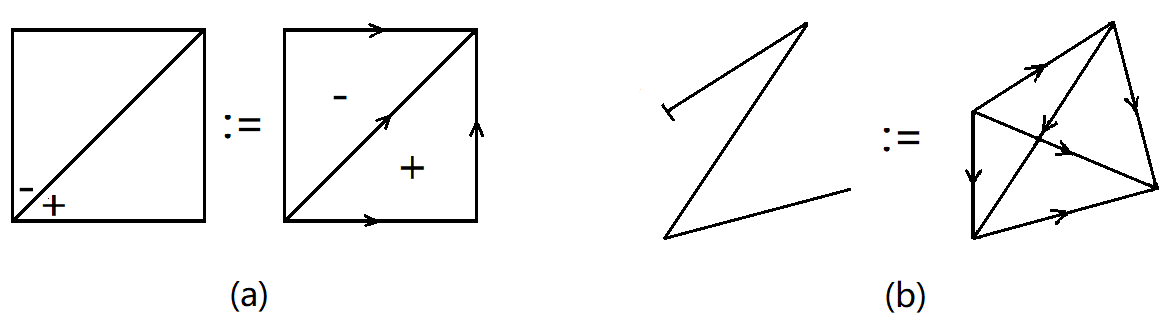}\\
  \caption{(a) An s-triangulation for a square; (b) a 3-simplex}\label{fig:notation}
\end{figure}

Since $\partial[ABBB]=[BBB]-[ABB]$, we have $[ABB]\doteq [BBB]$; similarly, $[BAB]\doteq [BBA]\doteq [BBB]$. It then follows from $\partial[ABCB]=[BCB]-[ACB]+[ABB]-[ABC]$ that
\begin{align}
[ACB]\equiv -[ABC].    \label{eq:inverse-triangle}
\end{align}
Consequently,
\begin{align}
([ABC]-[ADC])-([DAB]-[DCB])\equiv \partial[DABC],  \label{eq:calculus-square}
\end{align}
as illustrated by Fig. \ref{fig:calculus-square}.

\begin{figure}[h]
  \centering
  \includegraphics[width=0.4\textwidth]{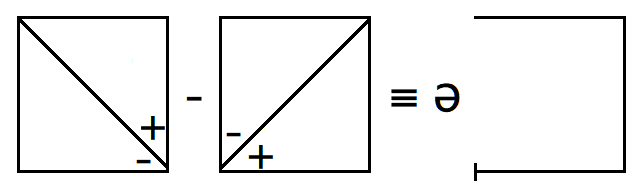}\\
  \caption{The identity (\ref{eq:calculus-square})}\label{fig:calculus-square}
\end{figure}

The assumption that $\alpha$ is strongly normalized implies the following:
\begin{align}
&\alpha(x,y,(xy)^{-1})=0; \label{eq:identity1}  \\
&\alpha(x,y,z)=\alpha((xyz)^{-1},x,y)=\alpha(xyz,(yz)^{-1},z)=\alpha(xy,z,(yz)^{-1})=-\alpha(x,yz,z^{-1});   \label{eq:identity2}
\end{align}
the second line is illustrated in Fig. \ref{fig:identity}.
These are easy to deduce and will be used implicitly from now on.

\begin{figure}[h]
  \centering
  \includegraphics[width=0.6\textwidth]{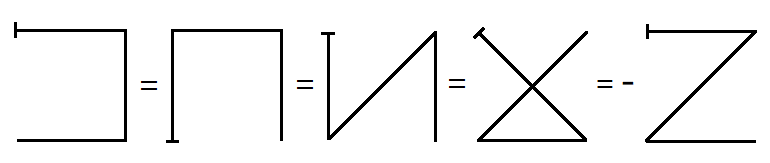}\\
  \caption{The reason for (\ref{eq:identity2})}\label{fig:identity}
\end{figure}

\subsubsection{Pair of pants $P=\Sigma_{0,3}$}

\begin{figure}[h]
  \centering
  \includegraphics[width=0.6\textwidth]{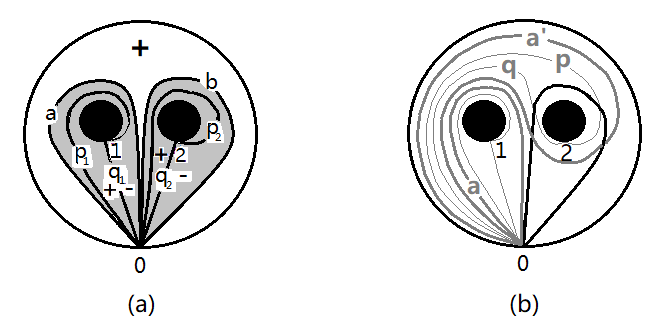}\\
  \caption{(a) The standard s-triangulation $\xi^{\rm st}_P$ of the pair of pants; (b) depicted in gray is the image of $\xi^{\rm st}_P$ under the twist $\varphi$} \label{fig:P}
\end{figure}

Let $[1^\frown]$ (resp. $[2^\frown]$) denote the clockwise loop from $1$ (resp. $2$) to itself. The other notations in the following lines deserve no explanation.
Let
\begin{align}
\xi^{\rm st}_P=[0\stackrel{q_1}\frown1^\frown]-[0\stackrel{a}{\frown}0\stackrel{q_1}\frown1]+[0\stackrel{q_2}\frown2^\frown]
-[0\stackrel{b}{\frown}0\stackrel{q_2}\frown2]+[0\stackrel{a}{\frown}0\stackrel{b}{\frown}0]. \label{eq:xi-P}
\end{align}
Let $\varphi:P\to P$ denote the clockwise twist, under which the two holes are interchanged.
The transformed s-triangulation is found to be
$$\varphi_\#\xi^{\rm st}_P=[0\stackrel{q}{\frown}2^\frown]-[0\stackrel{a'}{\frown}0\stackrel{q}{\frown}2]
+[0\stackrel{a'}{\frown}0\stackrel{a}{\frown}0]+[0\stackrel{q_1}\frown1^\frown]
-[0\stackrel{a}{\frown}0\stackrel{q_1}\frown1],$$
as shown in Fig. \ref{fig:P} (b). Then
\begin{align*}
\varphi_\#\xi^{\rm st}_P-\xi_P^{\rm st}&=
\partial\big([0\stackrel{q}{\frown}2\stackrel{q_2}{\frown}0\stackrel{p_2}{\frown}2]
-[0\stackrel{a}{\frown}0\stackrel{p_2}{\frown}2\stackrel{q_2}{\frown}0]
+[0\stackrel{a'}{\frown}0\stackrel{q}{\frown}2\stackrel{q_2}{\frown}0]    \\
&\ \ \ \ \ \ \ -[0\stackrel{p_2}{\frown}2\stackrel{q_2}{\frown}0\stackrel{q_2}{\frown}2]
-[0\stackrel{q_2}{\frown}2\stackrel{q_2}{\frown}0\stackrel{p_2}{\frown}2]
+[0\stackrel{q_2}{\frown}2\stackrel{q_2}{\frown}0\stackrel{q_2}{\frown}2]\big)   \\
&\equiv\partial\big([0\stackrel{q}{\frown}2\stackrel{q_2}{\frown}0\stackrel{p_2}{\frown}2]
-[0\stackrel{a}{\frown}0\stackrel{p_2}{\frown}2\stackrel{q_2}{\frown}0]
+[0\stackrel{a'}{\frown}0\stackrel{q}{\frown}2\stackrel{q_2}{\frown}0]\big).
\end{align*}
Let $\tau_{y_1,y_2}\in\mathfrak{B}^0(P)$ be the one given by $[1\stackrel{q_1}\frown 0]\mapsto 1$, $[2\stackrel{q_2}\frown 0]\mapsto y_1$ and $[j^\frown]\mapsto y_j$ for $j=1,2$, then
\begin{align}
\omega^{\tau_{y_1,y_2}}\big(\varphi_\#\xi^{\rm st}_P;\xi_P^{\rm st}\big)=-\alpha(y_1,y_1^{-1}y_2,y_1). \label{eq:cocycle-P}
\end{align}

\subsubsection{Disk with at least 3 holes removed}

Let $\xi_1,\xi_2$ be the two s-triangulations shown in Fig. \ref{fig:associator} (a), where the inner parts in light gray are similar as those in Fig. \ref{fig:P} (a).

For $x_1,x_2,x_3\in\Gamma$, let $\tau_{x_1,x_2,x_3}\in\mathfrak{B}^0(\Sigma_{0,4})$ send $[j^\frown]$ to $x_j$ and $[j\frown 0]$ to $1$ for $j=1,2,3$. Since
\begin{align*}
\xi_2-\xi_1=[0\stackrel{a}{\frown}0\stackrel{d_2}{\frown}0]+[0\stackrel{b}{\frown}0\stackrel{c}{\frown}0]
-[0\stackrel{d_1}{\frown}0\stackrel{c}{\frown}0]-[0\stackrel{a}{\frown}0\stackrel{b}{\frown}0] =\partial[0\stackrel{a}{\frown}0\stackrel{b}{\frown}0\stackrel{c}{\frown}0],
\end{align*}
we have
\begin{align}
\omega^{\tau_{x_1,x_2,x_3}}(\xi_2;\xi_1)=\alpha(x_1,x_2,x_3).  \label{eq:associator0}
\end{align}

\begin{figure}[h]
  \centering
  \includegraphics[width=0.75\textwidth]{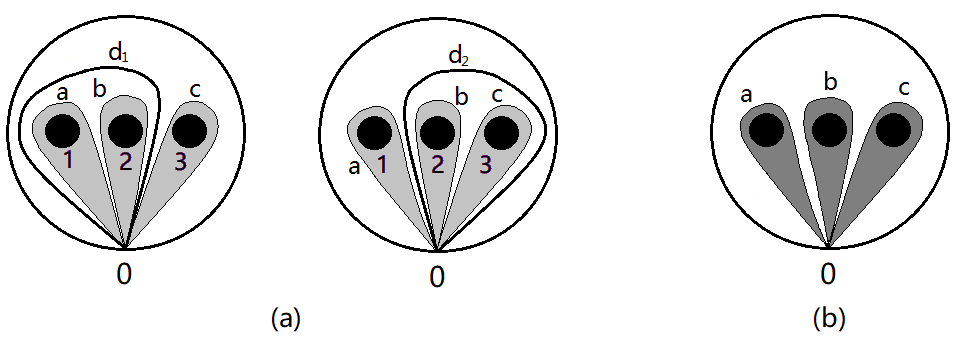}\\
  \caption{(a) $\xi_1,\xi_2\in[\Sigma_{0,4}]$; (b) the white part is $[0\stackrel{a}{\frown}0\stackrel{b}{\frown}0\stackrel{c}{\frown}0]$}\label{fig:associator}
\end{figure}

In this manner, if $r\ge 3$, then an s-triangulation of $\Sigma_{0,r+1}$ corresponds to a vertex in the associahedron $K_r$ (c.f. \cite{CSZ15}), and given two such s-triangulations $\xi,\xi'$, we may find $\zeta\in C_3(\Sigma_{0,r+1})$ with $\partial\zeta=\xi-\xi'$, which corresponds to a path in $K_r$. Then, letting $\tau_{x_1,\ldots,x_r}\in\mathfrak{B}^0(\Sigma_{0,r+1})$ send $[j^\frown]$ to $x_j$ and $[j\frown 0]$ to $1$ for $j=1,\ldots,r$, we can compute $\omega^{\tau_{x_1,\ldots,x_r}}(\xi;\xi')$ by successively using (\ref{eq:associator0}).
Call this value an {\it associator}; it is independent of the choice of $\zeta$. 

\subsubsection{Cylinder $C$}

We often draw a cylinder as a rectangle with a pair of opposite edges in double lines, to indicate that they are to be identified.

Let $\tilde{\xi}^n_C,\xi^n_C,\tilde{\xi}^{\rm st}_C,\xi^{\rm st}_C$ respectively denote the four s-triangulations in Fig. \ref{fig:equivalence} from left to right.
As an immediate consequence of (\ref{eq:inverse-triangle}), we have equivalences
\begin{align}
\tilde{\xi}^n_C\equiv\xi^n_C,  \qquad \tilde{\xi}^{\rm st}_C\equiv\xi^{\rm st}_C.   \label{eq:equivalence}
\end{align}

\begin{figure}[h]
  \centering
  \includegraphics[width=0.6\textwidth]{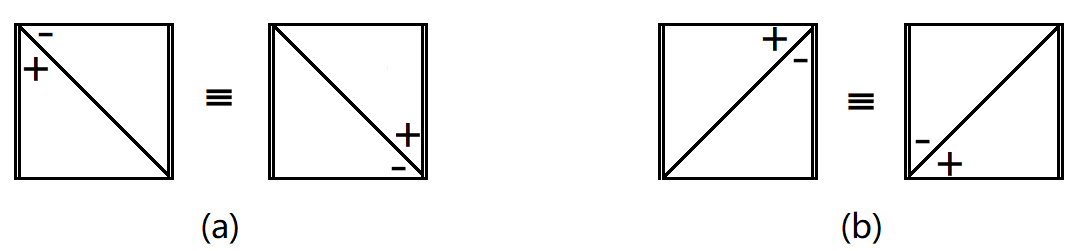}\\
  \caption{(a) $\tilde{\xi}^n_C\equiv\xi^n_C$; (b) $\tilde{\xi}^{\rm st}_C\equiv\xi^{\rm st}_C$} \label{fig:equivalence}
\end{figure}

Two cylinders can be glued into a new one: $C_1\cup C_2=C$, and a 3-chain $\zeta(C_1,C_2)$ whose boundary is
$\xi^{\rm st}_{C_1}+\xi^{\rm st}_{C_2}-\xi^{\rm st}_{C}$ can be found via Rule (I) shown in Fig. \ref{fig:difference}.

\begin{figure}[h]
  \centering
  \includegraphics[width=0.7\textwidth]{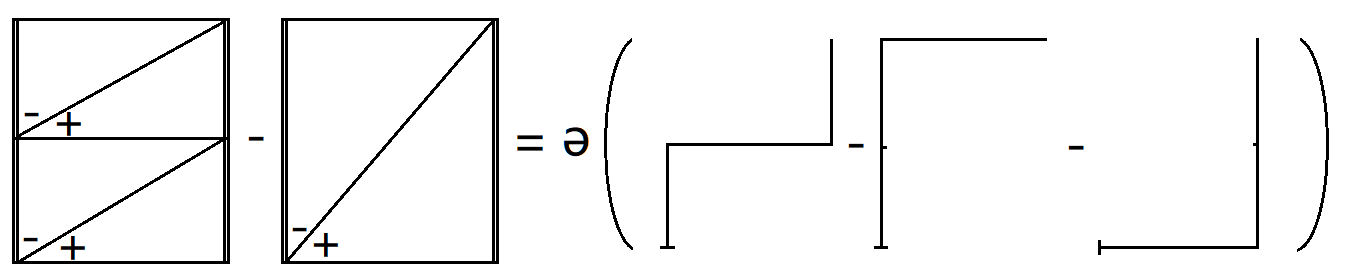}\\
  \caption{Rule (I): $\xi^{\rm st}_{C_1}+\xi^{\rm st}_{C_2}-\xi^{\rm st}_{C}=\partial\zeta(C_1,C_2)$}\label{fig:difference}
\end{figure}

\subsubsection{Torus $\mathbb{T}^2$}

Let $\pi:\mathbb{R}^2\to S^1\times S^1=\mathbb{T}^2$ be the universal covering. For $A_i\in\mathbb{Z}^2\subset\mathbb{R}^2$, $i=0,1,2$, use $[A_0A_1A_2]$ to denote the singular 2-simplex $\Delta^2\stackrel{\kappa}\to\mathbb{R}^2\stackrel{\pi}\to\mathbb{T}^2$, where $\kappa$ is the map extending $v_i\mapsto A_i$ linearly.
Similarly for singular 3-simplices in $\mathbb{T}^2$.

Write elements of $\mathbb{R}^2$ as column vectors. Let
\begin{align}
\xi^{\rm st}_{\mathbb{T}^2}&=[OW_1W_3]-[OW_2W_3]\in[\mathbb{T}^2], \\
\text{with} \qquad O&=\left(\begin{array}{cc} 0 \\ 0 \end{array}\right), \qquad
W_1=\left(\begin{array}{cc} 1 \\ 0 \end{array}\right), \qquad
W_2=\left(\begin{array}{cc} 0 \\ 1 \end{array}\right), \qquad
W_3=\left(\begin{array}{cc} 1 \\ 1 \end{array}\right).
\end{align}

Recall the well-known fact that the mapping class group $\mathcal{M}_1$ of $\mathbb{T}^2$ is isomorphic to ${\rm SL}(2,\mathbb{Z})$, which is generated by
$$\mathbf{s}=\left(\begin{array}{cc} 0 & 1 \\ -1 & 0 \end{array}\right), \qquad
\mathbf{t}=\left(\begin{array}{cc} 1 & 1 \\ 0 & 1 \end{array}\right).$$
For $\mathbf{a}\in{\rm SL}(2,\mathbb{Z})$, let $\hat{\mathbf{a}}:\mathbb{T}^2\to\mathbb{T}^2$ denote the homeomorphism induced by the left multiplication $\mathbf{a}:\mathbb{R}^2\to\mathbb{R}^2$. Then $\mathbf{a}\mapsto\hat{\mathbf{a}}$ sets up an isomorphism of ${\rm SL}(2,\mathbb{Z})$ onto $\mathcal{M}_1$.

For each $z\in\Gamma$, the map $\mathbb{Z}^3\to A$, $(a,b,c)\mapsto\alpha(z^a,z^b,z^c)$
is a 3-cocycle of $\mathbb{Z}$, so, due to $H^3(\mathbb{Z};A)=0$, there exists $f_z:\mathbb{Z}^2\to A$ such that
\begin{align}
\alpha(z^a,z^b,z^c)=f_z(b,c)-f_z(a+b,c)+f_z(a,b+c)-f_z(a,b)  \label{eq:fz}
\end{align}
for all $a,b,c\in\mathbb{Z}$.
Put
\begin{align}
\epsilon(z;a,b)=f_z(a,b)-f_z(b,a);
\end{align}
clearly, it is independent of the choice of $f_z$.

Setting $c=a$ in (\ref{eq:fz}), we obtain
\begin{align}
\alpha(z^a,z^b,z^a)=\epsilon(z;a,a+b)-\epsilon(z;a,b). \label{eq:alpha-epsilon}
\end{align}
Furthermore, in (\ref{eq:fz}), the case $b=0$ implies $f_z(a,0)=f_z(0,c)$, the case $c=-b=a$ implies $f_z(a,-a)=f_z(-a,a)$, and
\begin{align*}
a=u,\ b=-v,\ c=v\Rightarrow f_z(-v,v)-f_z(u-v,v)+f_z(u,0)-f_z(u,-v)=0, \\
a=v,\ b=-v,\ c=u\Rightarrow f_z(-v,u)-f_z(0,u)+f_z(v,u-v)-f_z(v,-v)=0.
\end{align*}
so that $\epsilon(z;u-v,v)=\epsilon(z;-v,u)$.
Hence
\begin{align}
\epsilon(z;a,b)=\epsilon(z;a+b,-a)=\epsilon(z;b,-a-b)=\epsilon(z;-a,-b). \label{eq:epsilon}
\end{align}


\begin{lem}
Let $\phi_z\in\mathfrak{B}^0(\mathbb{T}^2)$ be determined by $S^1\times 1\to 1$ and $1\times S^1\to z$. Then
\begin{align}
\omega^{\phi_z}\big(\hat{\mathbf{a}}_\#\xi^{\rm st}_{\mathbb{T}^2};\xi^{\rm st}_{\mathbb{T}^2}\big)
=\epsilon(z;c,d) \qquad \text{if} \quad \mathbf{a}=\left(\begin{array}{cc} a & b \\ c & d \end{array}\right).  \label{eq:genus1}
\end{align}
\end{lem}
\begin{proof}
The result is trivial when $\mathbf{a}$ is the identity matrix. The proof proceeds as showing: if (\ref{eq:genus1}) is true for $\mathbf{a}=\mathbf{b}$, then it is also true for
$$\mathbf{a}\mathbf{s}=\left(\begin{array}{cc} -b & a \\ -d & c \end{array}\right), \quad
\mathbf{a}\mathbf{s}^{-1}=\left(\begin{array}{cc} b & -a \\ d & -c \end{array}\right), \quad
\mathbf{a}\mathbf{t}=\left(\begin{array}{cc} a & a+b \\ c & c+d \end{array}\right), \quad
\mathbf{a}\mathbf{t}^{-1}=\left(\begin{array}{cc} a & b-a \\ c & d-c \end{array}\right).$$

\begin{figure}[h]
  \centering
  \includegraphics[width=0.55\textwidth]{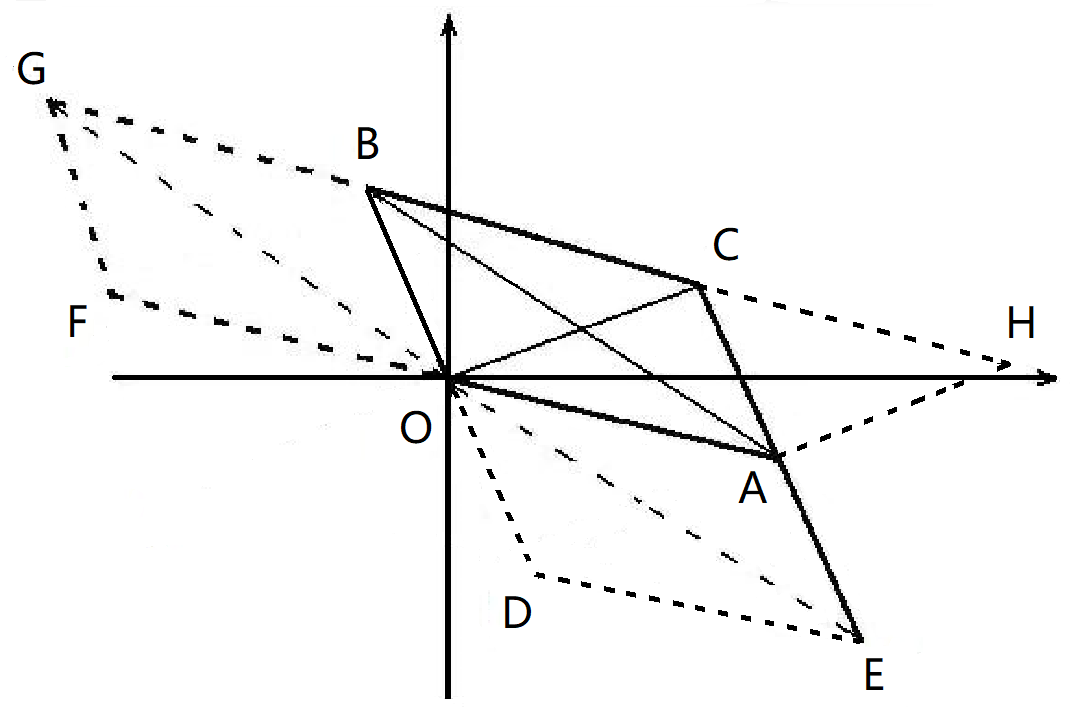}\\
  \caption{Cycle-cocycle calculus in genus 1}\label{fig:calculus-genus1}
\end{figure}

Let $A=(a,c), B=(b,d), C=(a+b,c+d)$. 
Assuming (\ref{eq:genus1}), we have (as in Fig. \ref{fig:calculus-genus1})
\begin{align*}
\hat{\mathbf{a}}_{\#}\xi^{\rm st}_{\mathbb{T}^2}&=[OAC]-[OBC]=[OAC]-[ACH], \\
(\hat{\mathbf{a}}\hat{\mathbf{s}})_\#\xi^{\rm st}_{\mathbb{T}^2}&=[ODE]-[OAE]=[BOA]-[BCA], \\
(\hat{\mathbf{a}}\hat{\mathbf{s}}^{-1})_{\#}\xi^{\rm st}_{\mathbb{T}^2}&=[OBG]-[OFG]=[ACB]-[AOB], \\
(\hat{\mathbf{a}}\hat{\mathbf{t}})_{\#}\xi^{\rm st}_{\mathbb{T}^2}&=[OAH]-[OCH], \\
(\hat{\mathbf{a}}\hat{\mathbf{t}}^{-1})_{\#}\xi^{\rm st}_{\mathbb{T}^2}&=[OAB]-[OGB]=[OAB]-[ABC].
\end{align*}

Computing directly, one obtains
\begin{align*}
(\hat{\mathbf{a}}\hat{\mathbf{s}})_{\#}\xi^{\rm st}_{\mathbb{T}^2}-\hat{\mathbf{a}}_{\#}\xi^{\rm st}_{\mathbb{T}^2}
=[BOA]-[BCA]-[OAC]+[OBC]\doteq\partial[OBAC],
\end{align*}
hence $\omega^{\phi_z}\big((\hat{\mathbf{a}}\hat{\mathbf{s}})_{\#}\xi^{\rm st}_{\mathbb{T}^2};\hat{\mathbf{a}}_{\#}\xi^{\rm st}_{\mathbb{T}^2}\big)=\alpha(z^d,z^{c-d},z^d)$,
which together with the assumption implies
\begin{align*}
&\omega^{\phi_z}((\hat{\mathbf{a}}\hat{\mathbf{s}})_{\#}\xi^{\rm st}_{\mathbb{T}^2};\xi^{\rm st}_{\mathbb{T}^2})=
\omega^{\phi_z}\big((\hat{\mathbf{a}}\hat{\mathbf{s}})_{\#}\xi^{\rm st}_{\mathbb{T}^2};\hat{\mathbf{a}}_{\#}\xi^{\rm st}_{\mathbb{T}^2}\big)+\omega^{\phi_z}\big(\hat{\mathbf{a}}_{\#}\xi^{\rm st}_{\mathbb{T}^2};\xi^{\rm st}_{\mathbb{T}^2}\big) \\
=\ &\alpha(z^d,z^{c-d},z^d)+\epsilon(z;c,d)\stackrel{(\ref{eq:alpha-epsilon})}
=\epsilon(z;d,c)-\epsilon(z;d,c-d)+\epsilon(z;c,d)\stackrel{(\ref{eq:epsilon})}=\epsilon(z;-d,c).
\end{align*}
So (\ref{eq:genus1}) holds for $\mathbf{a}\mathbf{s}$.

Similarly, we can also prove (\ref{eq:genus1}) for $\mathbf{a}\mathbf{s}^{-1}$, $\mathbf{a}\mathbf{t}^{\pm1}$. The reader may consult the proof of Lemma 3.4 in \cite{Ch16}.
\end{proof}

\section{A practical method for computations}

The main results are Algorithm \ref{alg} for computing $F(M,\rho)$ when $M$ is a link complement, and the formula (\ref{eq:main}) for $F(M,\rho)$ when $M$ is presented as a surgery along a link.

\subsection{Link complements}

Let $L=\sqcup_{i=1}^nL_i\subset S^3$ be a link, where the $L_i$'s are connected components. Take a sufficiently large solid cylinder ${\rm SC}=D^2\times[0,1]$ containing the tubular neighborhood $\mathcal{N}(L)$, and let $E_L=S^3-\mathcal{N}(L)$, $E'_L={\rm SC}-\mathcal{N}(L)$, then $E_L=E'_L\cup D^3$, where $D^3\cong S^3-{\rm SC}$ does not affect anything below. Suppose $L$ is presented as a diagram $\mathcal{L}=\sqcup_{i=1}^n\mathcal{L}_i$, with $\mathcal{L}_i$ corresponding to $L_i$, and suppose a representation $\rho_c:\pi_1(E_L)\to\Gamma$ is given via a {\it coloring} $c:\mathfrak{D}_{\mathcal{L}}\to\Gamma$ (where $\mathfrak{D}_{\mathcal{L}}$ is the set of directed arcs) fitting the Wirtinger presentation for $\pi_1(E_L)$.

\begin{figure}[h]
  \centering
  \includegraphics[width=0.48\textwidth]{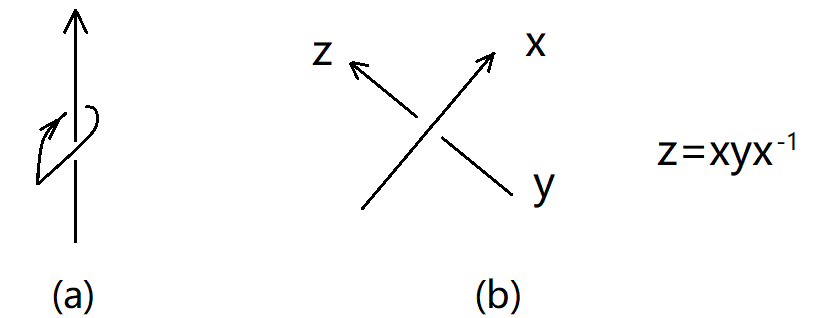}\\
  \caption{(a) each directed arc corresponds to an element of $\pi_1(E_L)$; (b) each crossing gives a relation}\label{fig:Wirtinger}
\end{figure}

\begin{conv}
\rm We adopt the ``over presentation" for $\pi_1(E_L)$ (see \cite{CF77} Chapter VI), as in Fig. \ref{fig:Wirtinger}.

Fix an orientation for $E_L$. For each $i$, fix an homeomorphism from $\mathbb{T}^2$ to the $\partial\mathcal{N}(L_i)$, and let $-\mathbb{T}^2_i$ stand for the $i$-th component of $\partial E_L$;
denote the image of $S^1\times 1$ by $\mathfrak{m}_i$ (the meridian), and denote that of $1\times S^1$ by $\mathfrak{l}_i$ (the longitude); label $\mathfrak{m}_i\cap\mathfrak{l}_i$ via a big dot on some arc of $\mathcal{L}_i$ and call it {\it basepoint}.
\end{conv}

\begin{figure}[h]
  \centering
  \includegraphics[width=0.9\textwidth]{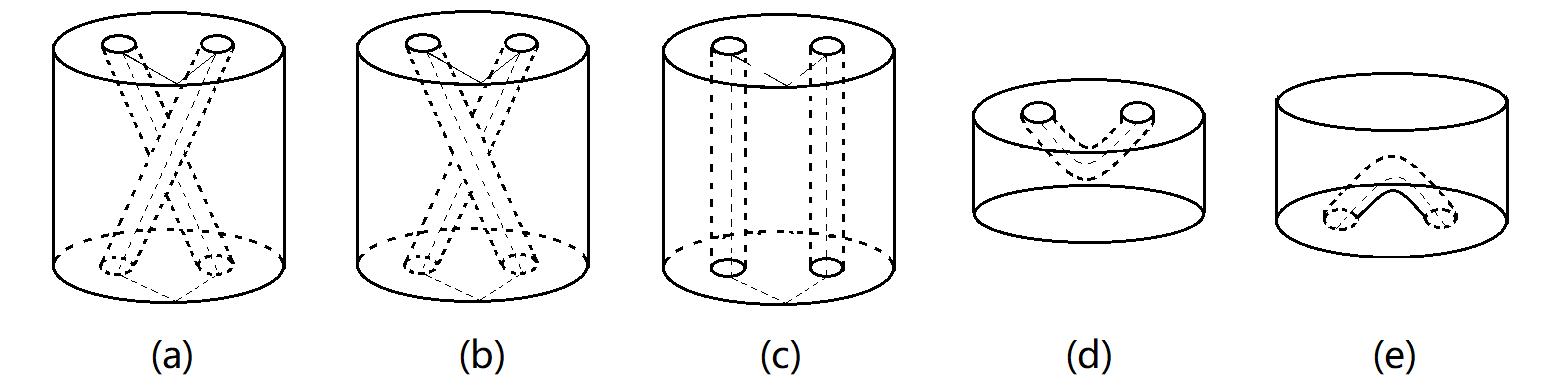}\\
  \caption{Basic pieces: (a) $Q(+)$; (b) $Q(-)$; (c) $Q(\|)=P\times [0,1]$; (d) $Q(\cup)$; (e) $Q(\cap)$} \label{fig:pieces}
\end{figure}

Use horizontal lines to cut $\mathcal{L}$ into simple pieces. The corresponding 3-dimensional picture is to use horizontal planes to decompose $E'_L$ into layers, each of which can be chopped into ``basic pieces" exhibited in Fig. \ref{fig:pieces}; note that $\Sigma_{0,k}\times[0,1]$ for $k\ge 4$ can be obtained by successively gluing $k-2$ copies of $Q(\|)$.
Recalling Convention \ref{conv:S}, choose for $E'_L$ the finite set consisting of the vertices, one for each circle.
From $c$ we can construct a representation (abusing the notation) $\rho_c\in\mathfrak{B}^0(E'_L)$ in an self-evident way, so that when restricted to $Q(+), Q(-)$ it is respectively in the form shown in Fig. \ref{fig:Qpm} (a),(c).

\begin{figure}[h]
  \centering
  \includegraphics[width=1.0\textwidth]{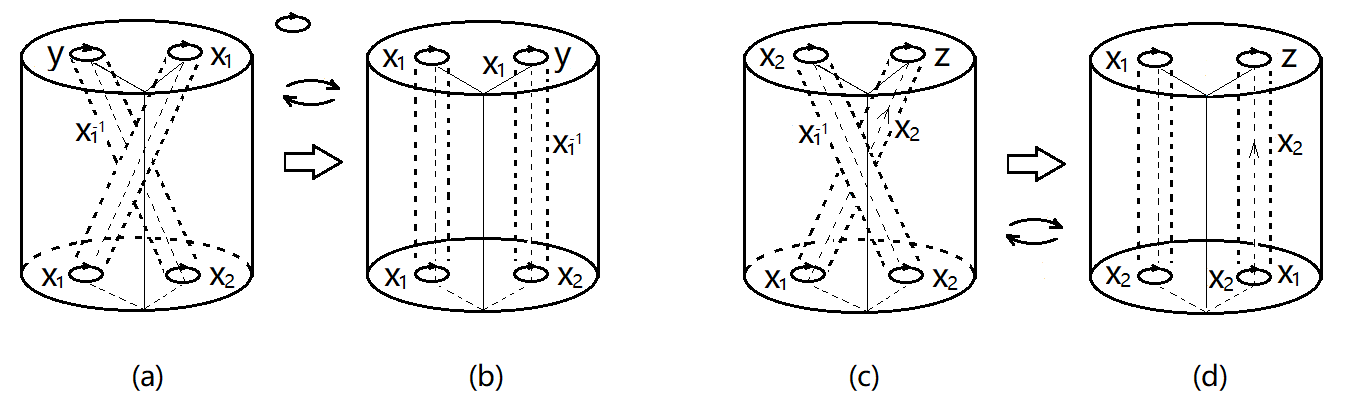}\\
  \caption{(a) $\rho^+_{x_1,x_2}\in\mathfrak{B}^0(Q(+))$, with $y=x_1x_2x_1^{-1}$; (b) $Q(+)$ can be transformed into $Q(\|)$ by a twist whose restriction on the upper boundary is $\varphi$; (c) $\rho^-_{x_1,x_2}\in\mathfrak{B}^0(Q(-))$, with $z=x_2^{-1}x_1x_2$; (d) $Q(-)$ can be transformed into $Q(\|)$ by a twist whose restriction on the lower boundary is $\varphi$}\label{fig:Qpm}
\end{figure}

For $Q(\pm)$, as parts of its boundary, let $P_{\rm u}$, $P_{\rm l}$ respectively denote the upper and lower pair of pants, let $C_{\rm o}$ denote the outer cylinder, and $C_{\rm i,1},C_{\rm i,2}$ the inner cylinders.

Let $\kappa^+_{x_1,x_2}\in\mathfrak{B}^0(Q(+))$ and $\varrho^+\in\mathfrak{B}^0(Q(\|))$ be determined by the assignments shown in Fig. \ref{fig:Qpm} (a), (b), respectively.
Then
\begin{align*}
&F(Q(+),\kappa^+_{x_1,x_2})(\xi^{\rm st}_{C_{\rm o}}+\xi^{\rm st}_{P_{\rm u}}-\xi^{\rm st}_{P_{\rm l}}-\xi^{\rm st}_{C_{\rm i,1}}-\xi^{\rm st}_{C_{\rm i,2}}) \\
=\ &F(Q(\|),\varrho^+)(\xi^{\rm st}_{C_{\rm o}}+\varphi_\#\xi^{\rm st}_{P_{\rm u}}-\xi^{\rm st}_{P_{\rm l}}-\xi^{\rm st}_{C_{\rm i,1}}-\xi^{\rm st}_{C_{\rm i,2}}) \\
=\ &F(Q(\|),\varrho^+)(\xi^{\rm st}_{C_{\rm o}}+\xi^{\rm st}_{P_{\rm u}}-\xi^{\rm st}_{P_{\rm l}}-\xi^{\rm st}_{C_{\rm i,1}}-\xi^{\rm st}_{C_{\rm i,2}})+\omega^{\tau_{x_1,y}}(\varphi_\#\xi^{\rm st}_{P_{\rm u}};\xi^{\rm st}_{P_{\rm u}})
=-\alpha(x_1,x_2x_1^{-1},x_1);
\end{align*}
in the last equality we use $\omega^{\tau_{x_1,y}}(\varphi_\#\xi^{\rm st}_{P_{\rm u}};\xi^{\rm st}_{P_{\rm u}})=-\alpha(x_1,x_2x_1^{-1},x_1)$ by (\ref{eq:cocycle-P}), and the equality
$$F(Q(\|),\varrho^+)(\xi^{\rm st}_{C_{\rm o}}+\xi^{\rm st}_{P_{\rm u}}-\xi^{\rm st}_{P_{\rm l}}-\xi^{\rm st}_{C_{\rm i,1}}-\xi^{\rm st}_{C_{\rm i,2}})=0$$
which is shown as follows: applying the ``prism operator" $\mathcal{P}$ (defined in \cite{Ha02} Page 112) to the homeomorphism $P_l\times I\to Q(\|)$, we get
$\partial\mathcal{P}(\xi^{\rm st}_{P_{\rm l}})=\xi^{\rm st}_{C_{\rm o}}+\xi^{\rm st}_{P_{\rm u}}-\xi^{\rm st}_{P_{\rm l}}-\xi^{\rm st}_{C_{\rm i,1}}-\xi^{\rm st}_{C_{\rm i,2}},$
and then
$$F(Q(\|),\varrho^+)(\xi^{\rm st}_{C_{\rm o}}+\xi^{\rm st}_{P_{\rm u}}-\xi^{\rm st}_{P_{\rm l}}-\xi^{\rm st}_{C_{\rm i,1}}-\xi^{\rm st}_{C_{\rm i,2}})=\alpha\big(\mathcal{P}(\xi^{\rm st}_{P_{\rm l}})\langle\varrho^+\rangle\big)=0.$$
Hence
\begin{align}
&F(Q(+),\kappa^+_{x_1,x_2})(\xi^{\rm st}_{C_{\rm o}}+\xi^{\rm st}_{P_{\rm u}}-\xi^{\rm st}_{P_{\rm l}}-\xi^{\rm st}_{C_{\rm i,1}}-\xi^{n}_{C_{\rm i,2}}) \nonumber \\
=\ &F(Q(+),\kappa^+_{x_1,x_2})(\xi^{\rm st}_{C_{\rm o}}+\xi^{\rm st}_{P_{\rm u}}-\xi^{\rm st}_{P_{\rm l}}-\xi^{\rm st}_{C_{\rm i,1}}-\xi^{\rm st}_{C_{\rm i,2}})+\omega(\xi^{\rm st}_{C_{\rm i,2}};\xi^{n}_{C_{\rm i,2}}) \nonumber \\
\stackrel{(\ref{eq:calculus-square})}=\ &-\alpha(x_1,x_2x_1^{-1},x_1)+\alpha(x_1,x_2x_1^{-1},x_1)=0. \label{eq:F(Q+)}
\end{align}

Let $\kappa^-_{x_1,x_2}\in\mathfrak{B}^0(Q(-))$ and $\varrho^-\in\mathfrak{B}^0(Q(\|))$ be determined by the assignments shown in Fig. \ref{fig:Qpm} (c), (d), respectively.
Then
\begin{align}
&F(Q(-),\kappa^-_{x_1,x_2})(\xi^{\rm st}_{C_{\rm o}}+\xi^{\rm st}_{P_{\rm u}}-\xi^{\rm st}_{P_{\rm l}}-\xi^{\rm st}_{C_{\rm i,1}}-\xi^{\rm st}_{C_{\rm i,2}}) \nonumber \\
=\ &F(Q(\|),\varrho^-)(\xi^{\rm st}_{C_{\rm o}}+\xi^{\rm st}_{P_{\rm u}}-\varphi_\#\xi^{\rm st}_{P_{\rm l}}-\xi^{\rm st}_{C_{\rm i,1}}-\xi^{\rm st}_{C_{\rm i,2}}) \nonumber \\
=\ &F(Q(\|),\varrho^-)(\xi^{\rm st}_{C_{\rm o}}+\xi^{\rm st}_{P_{\rm u}}-\xi^{\rm st}_{P_{\rm l}}-\xi^{\rm st}_{C_{\rm i,1}}
-\xi^{\rm st}_{C_{\rm i,2}})-\omega^{\tau_{x_2,x_1}}(\varphi_\#\xi^{\rm st}_{P_{\rm l}};\xi^{\rm st}_{P_{\rm l}}) \nonumber \\
=\ &\alpha(x_2^{-1},x_1,x_2)+\alpha(x_2,x_2^{-1}x_1,x_2)=0;   \label{eq:F(Q-)}
\end{align}
we have used $\omega^{\tau_{x_2,x_1}}(\varphi_\#\xi^{\rm st}_{P_{\rm l}};\xi^{\rm st}_{P_{\rm l}})=-\alpha(x_2,x_2^{-1}x_1,x_2)$ by (\ref{eq:cocycle-P}),
and the equality
$$F(Q(\|),\varrho^-)(\xi^{\rm st}_{C_{\rm o}}+\xi^{\rm st}_{P_{\rm u}}-\xi^{\rm st}_{P_{\rm l}}-\xi^{\rm st}_{C_{\rm i,1}}-\xi^{\rm st}_{C_{\rm i,2}})=\alpha(x_2^{-1},x_1,x_2)$$
which is obtained similarly as above, noting that the 3rd term in (\ref{eq:xi-P}) contributes $\alpha(x_2^{-1},x_1,x_2)$.

Furthermore, we can easily show that $Q(\cup)$ and $Q(\cap)$ contribute nothing; the details are omitted.

\begin{figure}[h]
  \centering
  \includegraphics[width=0.4\textwidth]{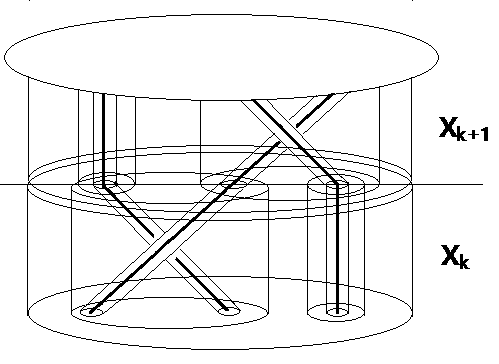}\\
  \caption{}\label{fig:associator0}
\end{figure}

Now $-\mathbb{T}_i^2$ appears to be glued from small cylinders, two for each crossing. However, the one arising from the overcrossing arc can be ``absorbed", due to Rule (I). Moreover, thanks to (\ref{eq:equivalence}), we are able to freely turn a square in a half-circle. Consequently, remembering (\ref{eq:F(Q+)}), (\ref{eq:F(Q-)}), we may find $\xi_{\mathcal{L},i}\in[-\mathbb{T}^2_i]$ by gluing small cylinders, one for each crossing according to Rule (II) (as presented in Fig. \ref{fig:crossing}). Then Rule (I) can be applied to compute $\omega\big(\xi_{\mathcal{L},i};\xi^{\rm st}_{\mathbb{T}^2_i}\big)$.

\begin{figure}[h]
  \centering
  \includegraphics[width=0.65\textwidth]{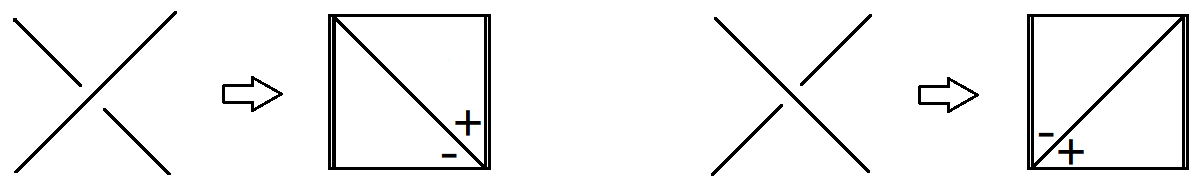}\\
  \caption{Rule (II), for crossings}\label{fig:crossing}
\end{figure}

Finally remained is another issue. When decomposing $E'_L$ into basic pieces, the two planar surfaces belonging to adjacent layers are usually triangulated differently, as illustrated in Fig. \ref{fig:associator0}. These account for associators.
To be precise, for the layers to be correctly glued, we must re-triangulate one of these two surfaces. Let $X_k$ denote the $k$-th layer, numbered from below to up, and let $\xi_k^{\rm u}$, $\xi^{\rm l}_k$, $\xi^{\rm i}_k$, $\xi^{\rm o}_k$ be the s-triangulations of the upper, lower, inner, outer boundary, respectively. We have computed
$F(X_k,\rho_k)(\xi^{\rm o}_k+\xi^{\rm u}_k-\xi^{\rm l}_k-\xi^{\rm i}_k)=0$, so
\begin{align*}
F(X_k,\rho_k)(\xi^{\rm o}_k+\xi^{\rm l}_{k+1}-\xi^{\rm l}_k-\xi^{\rm i}_k)
=\omega(\xi^{\rm l}_{k+1};\xi^{\rm u}_{k}). 
\end{align*}

All these are summarized to give
\begin{alg}\label{alg}
For each $i$, go ahead guided by the orientation of $L_i$, and draw a small triangulated cylinder according to Rule (II) whenever passing a crossing underneath. The result when back to the basepoint is an s-triangulation $\xi_{\mathcal{L},i}$ of $-\mathbb{T}_i^2$.

Use horizontal planes to decompose $E'_L$ into layers. Let $\mu_k^{\rm u}, \mu_k^{\rm l}$ ($1\le k\le m$) denote the s-triangulations of the $k$-th interface induced from the upper, lower layers, respectively.

Then
$F(E_L,\rho_c)\big(\sqcup_{i=1}^n-\xi^{\rm st}_{\mathbb{T}^2_i}\big)=\delta_{\mathcal{L},c}+\theta_{\mathcal{L},c}$,
with
\begin{align}
\delta_{\mathcal{L},c}=\sum\limits_{i=1}^n\omega\big(\xi_{\mathcal{L},i};\xi^{\rm st}_{\mathbb{T}^2_i}\big); \qquad \theta_{\mathcal{L},c}=\sum\limits_{k=1}^m\theta_k, \quad \theta_k=\omega(\mu^{\rm u}_k;\mu^{\rm l}_k).
\end{align}
\end{alg}

\begin{figure}[h]
  \centering
  \includegraphics[width=5cm]{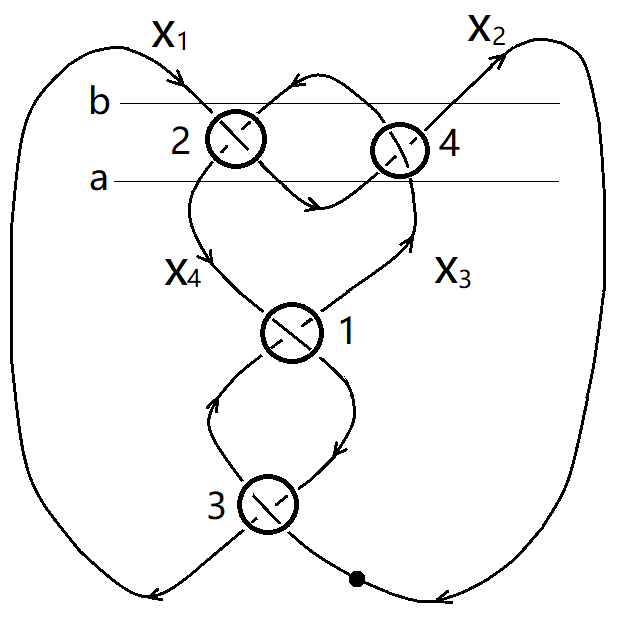}\\
  \caption{The diagram $\mathcal{K}$ for $4_1$, with a basepoint chosen and crossings numbered}\label{fig:eight}
\end{figure}

\begin{figure}[h]
  \centering
  \includegraphics[width=0.8\textwidth]{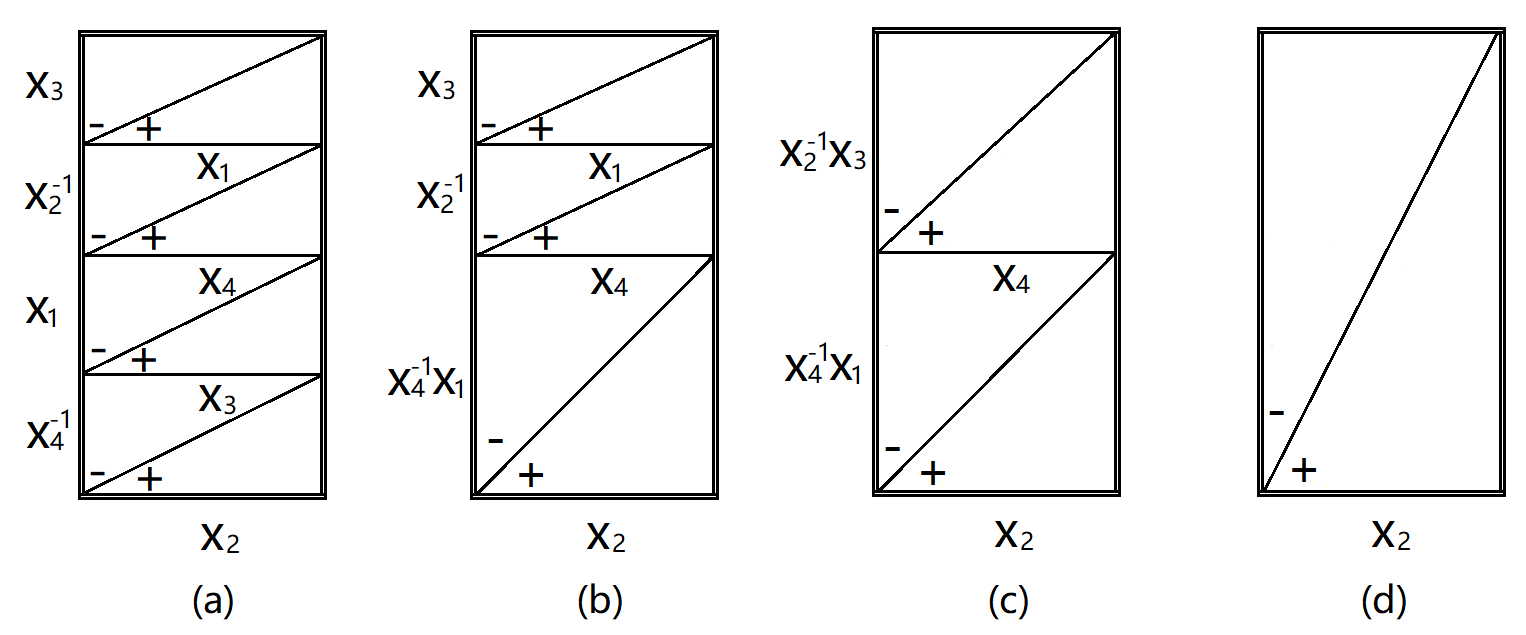}\\
  \caption{(a) $\xi_{\mathcal{K}}$; (d) $\xi^{\rm st}_{\mathbb{T}^2}$; (b), (c): intermediate s-triangulations}\label{fig:contribution-eight}
\end{figure}

\begin{exmp}
\rm Let $K$ be the figure eight knot, for which a diagram $\mathcal{K}$ together with a coloring $c$ is given in Fig. \ref{fig:eight}.

Referring to Fig. \ref{fig:contribution-eight} and successively applying Rule (I),
we get
\begin{align}
\delta_{\mathcal{K},c}=\ &\alpha(x_4^{-1},x_3,x_1)-\alpha(x_4^{-1},x_1,x_4)-\alpha(x_2,x_4^{-1},x_1) \nonumber \\
&+\alpha(x_2^{-1},x_1,x_3)-\alpha(x_2^{-1},x_3,x_2)-\alpha(x_4,x_2^{-1},x_3) \nonumber \\
&+\alpha(x_4^{-1}x_1,x_4,x_2^{-1}x_3)-\alpha(x_4^{-1}x_1,x_2^{-1}x_3,x_2)-\alpha(x_2,x_4^{-1}x_1,x_2^{-1}x_3).
 \label{eq:8-delta}
\end{align}
From Fig. \ref{fig:associator-eight}, using (\ref{eq:associator0}) twice, we see that the associator at level $a$ is
$$\theta_a=-\alpha(x_4^{-1},x_1^{-1},x_1)+\alpha(x_4^{-1}x_1^{-1},x_1,x_3)=\alpha(x_4^{-1}x_1^{-1},x_1,x_3).$$
Similarly, $\theta_b=-\alpha(x_1^{-1}x_3^{-1},x_3,x_2)$. The associators at the other levels all vanish. Hence
\begin{align}
\theta_{\mathcal{K},c}=\alpha(x_4^{-1}x_1^{-1},x_1,x_3)-\alpha(x_1^{-1}x_3^{-1},x_3,x_2).  \label{eq:8-theta}
\end{align}

Thus $F(E_K,\rho_c)(-\xi^{\rm st}_{\mathbb{T}^2})$ equals the sum of the right-hand-sides of (\ref{eq:8-delta}) and (\ref{eq:8-theta}).
\end{exmp}

\begin{figure}[h]
  \centering
  \includegraphics[width=0.7\textwidth]{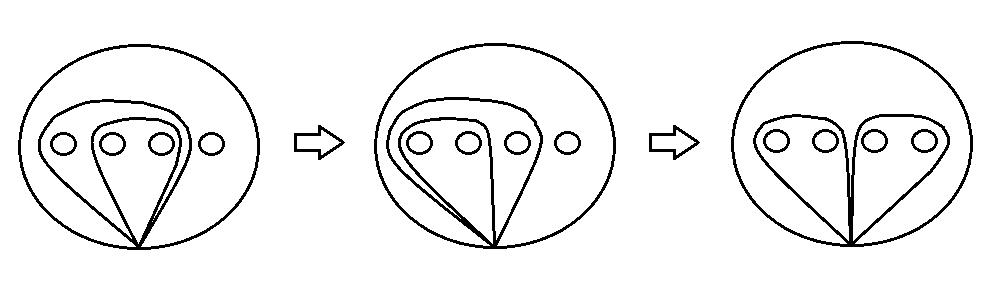}\\
  \caption{At level $a$, the transformation from $\mu^{\rm l}_a$ to $\mu^{\rm u}_a$, where the representation is given by $[1^\frown]\mapsto x_4^{-1}$, $[2^\frown]\mapsto x_1^{-1}$, $[3^\frown]\mapsto x_1$, $[4^\frown]\mapsto x_3$, and $[j\frown 0]\mapsto 1$ for $j=1,\ldots,4$}\label{fig:associator-eight}
\end{figure}

\begin{figure}[h]
  \centering
  \includegraphics[width=0.8\textwidth]{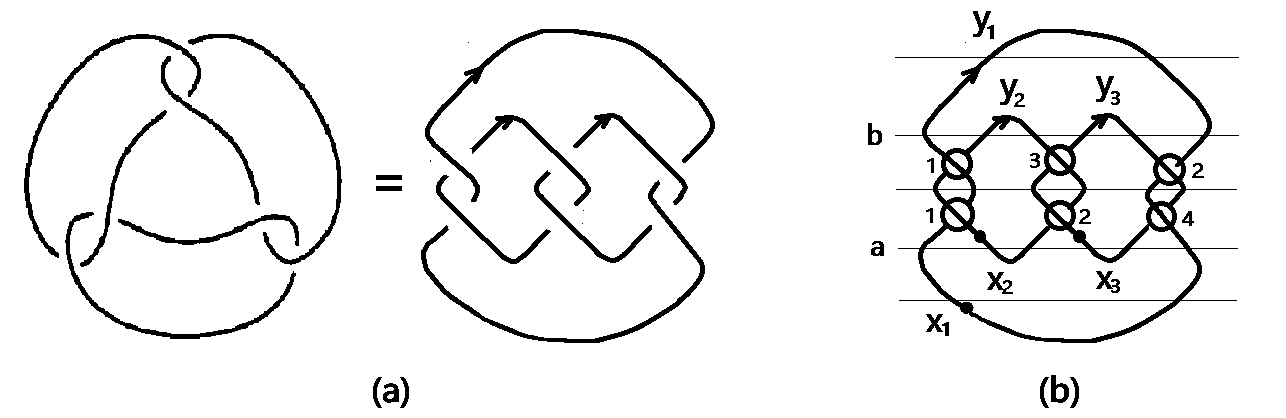}\\
  \caption{(a) The 3-chain link; (b) the second diagram $\mathcal{L}$ is used for computation}\label{fig:3-chain}
\end{figure}
\begin{figure}[h]
  \centering
  \includegraphics[width=0.5\textwidth]{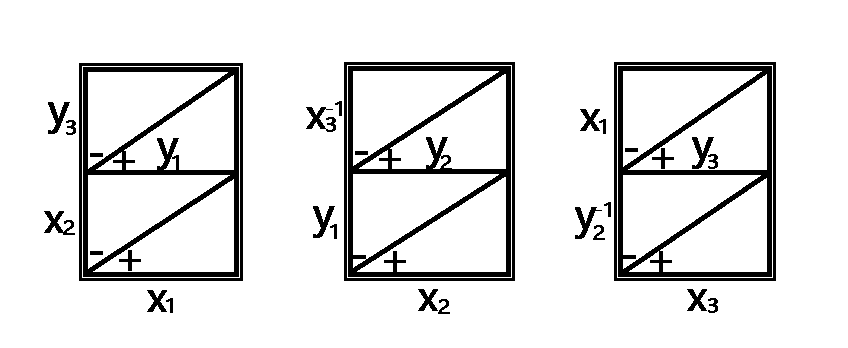}\\
  \caption{$\xi_{\mathcal{L},i}$, $i=1,2,3$}\label{fig:contribution-3-chain}
\end{figure}
\begin{figure}[h]
  \centering
  \includegraphics[width=0.6\textwidth]{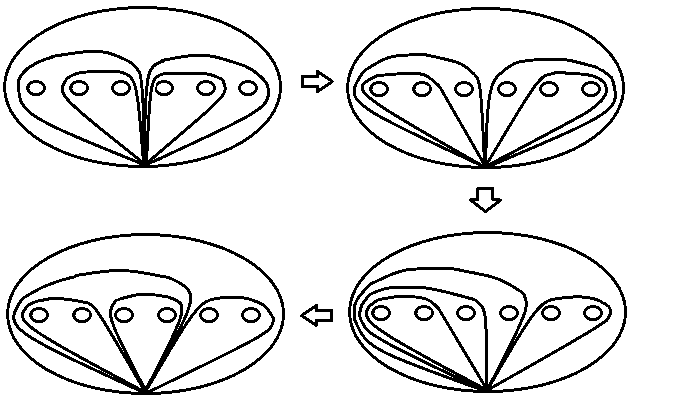}\\
  \caption{At level $a$, the transformation from $\mu^{\rm l}_a$ to $\mu^{\rm u}_a$, where the representation is given by $[1^\frown]\mapsto x_1$, $[2^\frown]\mapsto x_2$, $[3^\frown]\mapsto x_2^{-1}$, $[4^\frown]\mapsto x_3$, $[5^\frown]\mapsto x_3^{-1}$, $[6^\frown]\mapsto x_1^{-1}$, and $[j\frown 0]\mapsto 1$ for $j=1,\ldots,6$}\label{fig:associator-3-chain}
\end{figure}

\begin{exmp}
\rm Let $L$ be the 3-chain link, which is also the $(2,2,2)$-pretzel link. A diagram $\mathcal{L}$ together with a coloring $c$ is given in Fig. \ref{fig:3-chain} (b).

Referring to Fig. \ref{fig:contribution-3-chain} and applying Rule (I), we obtain
\begin{align}
\delta_{\mathcal{L},c}=\ &\alpha(x_2,y_1,y_3)-\alpha(x_2,y_3,x_1)-\alpha(x_1,x_2,y_3) \nonumber \\
&+\alpha(y_1,y_2,x_3^{-1})-\alpha(y_1,x_3^{-1},x_2)-\alpha(x_2,y_1,x_3^{-1})  \nonumber \\
&+\alpha(y_2^{-1},y_3,x_1)-\alpha(y_2^{-1},x_1,x_3)-\alpha(x_3,y_2^{-1},x_1).  \label{eq:3chain-delta}
\end{align}

Referring to Fig. \ref{fig:associator-3-chain}, we have
\begin{align*}
\theta_a&=(\alpha(x_3,x_3^{-1},x_1^{-1})-\alpha(x_1,x_2,x_2^{-1}))-\alpha(x_1,x_3,x_3^{-1}x_1^{-1})
+\alpha(x_1x_2,x_2^{-1},x_3) \\
&=\alpha(x_1x_2,x_2^{-1},x_3).
\end{align*}
Similarly,
$\theta_b=-\alpha(y_1y_2,y_2^{-1},y_3).$
The associators at the other levels all vanish.
So
\begin{align}
\theta_{\mathcal{L},c}=\alpha(y_1y_2,y_2^{-1},y_3)-\alpha(x_1x_2,x_2^{-1},x_3).  \label{eq:3chain-theta}
\end{align}

Now $F(E_L,\rho_c)\big(\sqcup_{i=1}^3-\xi^{\rm st}_{\mathbb{T}^2_i}\big)$ equals the sum of the right-hand-sides of (\ref{eq:3chain-delta}) and (\ref{eq:3chain-theta}).
\end{exmp}

\subsection{Closed 3-manifolds with a surgery presentation}

Consider the closed 3-manifold resulting from a surgery along a link $L$:
$$M=M(L;p_1/q_1,\ldots,p_n/q_n):=E_L\cup_{\sqcup_{i=1}^n[p_i/q_i]}(\sqcup_{i=1}^n{\rm ST}_i), \qquad
{\rm ST}_i=D^2\times S^1,$$
where $[p_i/q_i]=\widehat{\mathbf{a}_i}:\mathbb{T}^2\to\mathbb{T}^2$, with
$\mathbf{a}_i=\left(\begin{array}{cc} p_i & p'_i \\ q_i & q'_i \end{array}\right)\in{\rm SL}(2,\mathbb{Z})$ for some integrs $p'_i,q'_i$ that are irrelevant; the $i$-th solid torus ${\rm ST}_i$ is glued onto $E_L$ so that
$[p_i/q_i](S^1\times 1)=\mathfrak{m}_i^{p_i}\mathfrak{l}_i^{q_i}$.

Given $\rho\in\mathfrak{B}(M)$, let $\rho_L=\rho|_{E_L}$, and $\rho_i=\rho|_{{\rm ST}_i}$.
Applying Lemma \ref{lem:glue} to $X=E_L\sqcup(\sqcup_{i=1}^n{\rm ST}_i)$, $Y_1=\sqcup_{i=1}^n-\mathbb{T}^2_i$ and $Y_2=\partial E_L$, we deduce
\begin{align*}
F(M,\rho)&=F(E_L,\rho_L)\big(\sqcup_{i=1}^n-\xi^{\rm st}_{\mathbb{T}^2_i}\big)
+\sum\limits_{i=1}^nF({\rm ST}_i,\rho_i)\big([p_i/q_i]^{-1}_{\#}\xi^{\rm st}_{\mathbb{T}^2_i}\big)  \\
&=F(E_L,\rho_L)\big(\sqcup_{i=1}^n-\xi^{\rm st}_{\mathbb{T}^2_i}\big)
+\sum\limits_{i=1}^n\Big(F({\rm ST}_i,\rho_i)\big(\xi^{\rm st}_{\mathbb{T}^2_i}\big)
+\omega^{\phi_{z^i}}\big([p_i/q_i]^{-1}_{\#}\xi^{\rm st}_{\mathbb{T}^2_i};\xi^{\rm st}_{\mathbb{T}^2_i}\big)\Big),
\end{align*}
where $z_i=\rho\big(\mathfrak{m}_i^{p'_i}\mathfrak{l}_i^{q'_i}\big)$, which can be characterized by $z_i^{p_i}=\rho(\mathfrak{l}_i)$ and $z_i^{-q_i}=\rho(\mathfrak{m}_i)$. 
Thus, using (\ref{eq:genus1}), we obtain
\begin{align}
F(M,\rho)=F(E_L,\rho_L)\big(\sqcup_{i=1}^n-\xi^{\rm st}_{\mathbb{T}^2_i}\big)
+\sum\limits_{i=1}^n\epsilon(z_i;-q_i,p_i).   \label{eq:main}
\end{align}

\end{document}